\documentclass[graybox]{svmult}
\usepackage{type1cm}
\usepackage{makeidx}
\usepackage{graphicx}
\usepackage{multicol}
\usepackage[bottom]{footmisc}
\usepackage{newtxtext} 
\usepackage[varvw]{newtxmath}
\makeindex

\usepackage{algpseudocode}
\usepackage{algorithm}
\algrenewcommand\algorithmicrequire{\textbf{Input:}}
\usepackage{tikz}

\begin{document}
\title*{On Bounding and Approximating Functions of Multiple Expectations using Quasi-Monte Carlo}
\titlerunning{On Bounding and Approximating Functions of Multiple Expectations using QMC}
\authorrunning{Sorokin \& Rathinavel}
\author{Aleksei G. Sorokin \and Jagadeeswaran Rathinavel}
\institute{
    Aleksei G. Sorokin \at Department of Applied Mathematics, Illinois Institute of Technology,\\ RE 220, 10 W.\ 32$^{\text{nd}}$ St., Chicago, IL 60616 \email{asorokin@hawk.iit.edu}
\and
    R. Jagadeeswaran \at Department of Applied Mathematics, Illinois Institute of Technology,\\ RE 220, 10 W.\ 32$^{\text{nd}}$ St., Chicago, IL 60616 \email{jrathin1@iit.edu}; and \\Wi-Tronix LLC, 631 E Boughton Rd, Suite 240, Bolingbrook, IL 60440}
\maketitle

\abstract{
Monte Carlo and Quasi-Monte Carlo methods present a convenient approach for approximating the expected value of a random variable. Algorithms exist to adaptively sample the random variable until a user defined absolute error tolerance is satisfied with high probability. This work describes an extension of such methods which supports adaptive sampling to satisfy general error criteria for functions of a common array of expectations. Although several functions involving multiple expectations are being evaluated, only one random sequence is required, albeit sometimes of larger dimension than the underlying randomness. These enhanced Monte Carlo and Quasi-Monte Carlo algorithms are implemented in the QMCPy Python package with support for economic and parallel function evaluation. We exemplify these capabilities on problems from machine learning and global sensitivity analysis.
}

\newpage

\section{Introduction}

Theoretical developments, stopping criteria, and implementations of both Monte Carlo (MC) and Quasi-Monte Carlo (QMC) methods often focus on approximating a \emph{QOI} (quantity of interest) which is a scalar mean  $\mu = \mathbb{E}[f(\boldsymbol{X})]$ with $\boldsymbol{X} \sim \mathcal{U}[0,1]^d$ and integrand $f: [0,1]^d \to \mathbb{R}$. However, in many cases the QOI $\boldsymbol{s} \in \mathbb{R}^{\boldsymbol{d}_{\boldsymbol{s}}}$ is formulated as a more complicated function of an array \emph{mean} $\boldsymbol{\mu} = \mathbb{E}[\boldsymbol{f}(\boldsymbol{X})] \in \mathbb{R}^{\boldsymbol{d}_{\boldsymbol{\mu}}}$. Here $\boldsymbol{\mu}$ is a multi-dimensional array, which we simply call an array, with shape vector $\boldsymbol{d}_{\boldsymbol{\mu}}$ e.g. $\boldsymbol{\mu} \in \mathbb{R}^{(2,3)}$ indicates $\boldsymbol{\mu}$ is a $2 \times 3$ matrix. Similarly, we allow $\boldsymbol{s}$ to be an array with shape $\boldsymbol{d}_{\boldsymbol{s}}$. The integrand is now $\boldsymbol{f}: [0,1]^{d} \to \mathbb{R}^{\boldsymbol{d}_{\boldsymbol{\mu}}}$ with $\boldsymbol{X} \sim \mathcal{U}[0,1]^d$ as before. 

The QOI array is formulated from the mean array via a function $\boldsymbol{C}: \mathbb{R}^{\boldsymbol{d}_{\boldsymbol{\mu}}} \to \mathbb{R}^{\boldsymbol{d}_{\boldsymbol{s}}}$ so that $\boldsymbol{s} = \boldsymbol{C}(\boldsymbol{\mu})$.
Example QOI arrays include
\begin{itemize}
    \item an $(a \times b)$ mean matrix where $\boldsymbol{C}$ is the identity and $\boldsymbol{d}_{\boldsymbol{s}} = \boldsymbol{d}_{\boldsymbol{\mu}} = (a,b)$,
    \item a Bayesian posterior mean where $s = C(\mu_1,\mu_2) = \mu_1/\mu_2$, $d_s = 1$, and $d_{\boldsymbol{\mu}} = 2$, 
    \item $c$ closed and total sensitivity indices requiring $\boldsymbol{d}_{\boldsymbol{s}} = (2,c)$ and $\boldsymbol{d}_{\boldsymbol{\mu}} = (2,3,c)$ to formulate $s_{ij} = C_{ij}(\boldsymbol{\mu}) =  \mu_{i3j}/(\mu_{i2j}-\mu_{i1j}^2)$ for $i \in \{1,2\}$, $j \in \{1,\dots,c\}$.
\end{itemize}
These examples are further detailed in Section \ref{SoRa_sec:examples}.

This article generalizes \cite{adaptive_qmc} to develop Algorithm \ref{SoRa_algo:MCStoppingCriterion} which 
\begin{enumerate}
    \item produces bounds $[\boldsymbol{s}^-,\boldsymbol{s}^+]$ on QOI $\boldsymbol{s}$ which hold with elementwise uncertainty below a user specified threshold array $\boldsymbol{\alpha}^{(\boldsymbol{s})} \in (0,1)^{\boldsymbol{d}_{\boldsymbol{s}}}$, 
    \item computes an optimal QOI approximation $\hat{\boldsymbol{s}}$ based on bounds $[\boldsymbol{s}^-,\boldsymbol{s}^+]$, a user specified error metric, and user specified error tolerance,
    \item repeats with increasing sample sizes until the stopping criterion is satisfied. 
\end{enumerate}
The algorithm utilizes existing (Q)MC methods that, given an appropriate set of sampling nodes and their corresponding function evaluations, produce bounds $[\boldsymbol{\mu}^-,\boldsymbol{\mu}^+]$ on the mean $\boldsymbol{\mu}$ that hold with elementwise uncertainty below a derived threshold array $\boldsymbol{\alpha}^{(\boldsymbol{\mu})} \in (0,1)^{\boldsymbol{d}_{\boldsymbol{\mu}}}$. A dependency mapping from $\boldsymbol{s}$ to $\boldsymbol{\mu}$ is used to derive $\boldsymbol{\alpha}^{(\boldsymbol{\mu})}$ from $\boldsymbol{\alpha}^{(\boldsymbol{s})}$. Interval arithmetic functions are used to propagate mean bounds $[\boldsymbol{\mu}^-,\boldsymbol{\mu}^+]$ to QOI bounds $[\boldsymbol{s}^-,\boldsymbol{s}^+]$. These interval arithmetic functions are derived from $\boldsymbol{C}$ and problem specific QOI restrictions. When exiting approximations to QOI are sufficiently accurate, the dependency function may tell the algorithm that certain outputs of $\boldsymbol{f}$ are not necessary to evaluate, a principal we call \emph{economic evaluation}. 

These enhancements are adapted to a number of QMC algorithms in the open source QMCPy Python package \cite{QMCPy} which is distributed on both GitHub and PyPI. The implementations incorporate shared samples, parallel computation, and economic evaluation to efficiently find bounds and approximations satisfying flexible user specifications.

The remainder of the article is organized as follows. Section \ref{SoRa_sec:MCM} differentiates how MC and QMC approximate a scalar mean $\mu$. A more detailed account of this mature field is available in \cite{niederreiter1992random}. Section \ref{SoRa_sec:Existing_QMC_Methods} outlines some (Q)MC algorithms to infer bounds $[\mu^-,\mu^+]$ on a scalar mean $\mu$ with uncertainty below a specified threshold $\alpha^{(\mu)}$. In Section \ref{SoRa_sec:comb_sol_approx} we consider the case where $\boldsymbol{\mu}$ is an array and $s$ is a scalar. Here we describe how to set array $\boldsymbol{\alpha}^{(\boldsymbol{\mu})}$ based on scalar $\alpha^{(s)}$ and how to propagate bounds $[\boldsymbol{\mu}^-,\boldsymbol{\mu}^+]$ on $\boldsymbol{\mu}$ to bounds $[s^-,s^+]$ on $s$ so both hold with uncertainty below $\alpha^{(s)}$. Section \ref{SoRa_sec:opt_comb_sol_sc} derives a stopping criterion for adaptive sampling and optimal approximation $\hat{s}$ of scalar QOI $s$. Both the approximation and stopping criterion are based on $[s^-,s^+]$, a user-specified error metric, and a user-specified error threshold. Considerations for extending to array QOI $\boldsymbol{s}$, including a generalized method for setting $\boldsymbol{\alpha}^{(\boldsymbol{\mu})}$ and a strategy for economic evaluation, are discussed in Section \ref{SoRa_sec: Vectorized Implementation} before presenting the unifying Algorithm \ref{SoRa_algo:MCStoppingCriterion}. Section \ref{SoRa_sec:examples} gives examples from machine learning and sensitivity analysis before Section \ref{SoRa_sec:conclusions} discusses conclusions and future work.   

\section{Monte Carlo and Quasi-Monte Carlo Methods} \label{SoRa_sec:MCM}

(Q)MC methods are well-suited to approximate a scalar mean $\mu = \mathbb{E}[f(\boldsymbol{X})]$ with $\boldsymbol{X} \sim \mathcal{U}[0,1]^d$ and integrand $f: [0,1]^{d} \to \mathbb{R}$. A change of variables may be necessary to ensure $\boldsymbol{X}$ is standard uniform, see \cite{QMCSoftware} for details and default transforms implemented in QMCPy. 
(Q)MC methods often approximate $\mu$ by the sample average of $f$ evaluated at nodes $\boldsymbol{x}_1,\dots,\boldsymbol{x}_n \in [0,1]^d$. We denote this approximation by  
\begin{equation}
    \label{SoRa_eq:mcapprox}
    \hat{\mu} = \frac{1}{n}\sum_{i=1}^n f(\boldsymbol{x}_i) \approx \mathbb{E}[f(\boldsymbol{X})] = \mu. 
\end{equation}
MC methods choose the sampling nodes to be independent and identically distributed (IID), that is $\boldsymbol{x}_1,\dots,\boldsymbol{x}_n \overset{\text{\tiny IID}}{\sim} \mathcal{U}[0,1]^{d}$. For MC, the absolute approximation error $\lvert \mu - \hat{\mu} \rvert$ is $\mathcal{O}(n^{-1/2})$. 

QMC methods choose the sampling nodes in a dependent manner to improve uniformity. Discrepancy measures quantify how close the empirical distribution of $\{\boldsymbol{x}_i\}_{i=1}^n$ is to the standard uniform distribution. The Koksma-Hlawka inequality bounds the absolute approximation error by the star discrepancy of $\{\boldsymbol{x}_i\}_{i=1}^n$ times the variation of $f$ in the sense of Hardy and Krause \cite{dick2013high}. Other discrepancy-variation pairings are also available, see \cite{hickernell1998generalized} for an overview. While it is often impractical to determine if $f$ has bounded variation, such inequalities indicate that using low discrepancy (LD) sequences in place of IID sequences can improve performance for nicely behaved $f$. A number of LD sequences exist that achieve a discrepancy of $\mathcal{O}(n^{-1+\delta})$ for any $\delta > 0$. Such LD sequences are the hallmark of QMC methods. When $f$ has bounded variation this rate upper bounds the absolute error of QMC methods, a significant improvement over the rate for MC methods. 

The QMC methods in this article utilize randomized extensible LD sequences. Randomization ensures, with probability $1$, that $\boldsymbol{x}_1,\dots,\boldsymbol{x}_n \in (0,1)^d$ and that the LD sequence does not badly match the integrand. Extensibility enables algorithms to adaptively increase the number of samples required to meet the stopping criterion without needing to discard previous function evaluations. Digital sequences and integration lattices are two popular choices for LD sequences. Constructions exist for both that support randomization and extensible designs. Figure \ref{SoRa_fig:ld_seqs} contrasts IID points with LD sequences in base $2$.

\begin{figure}[t]
    \centering
    \includegraphics[width=\textwidth]{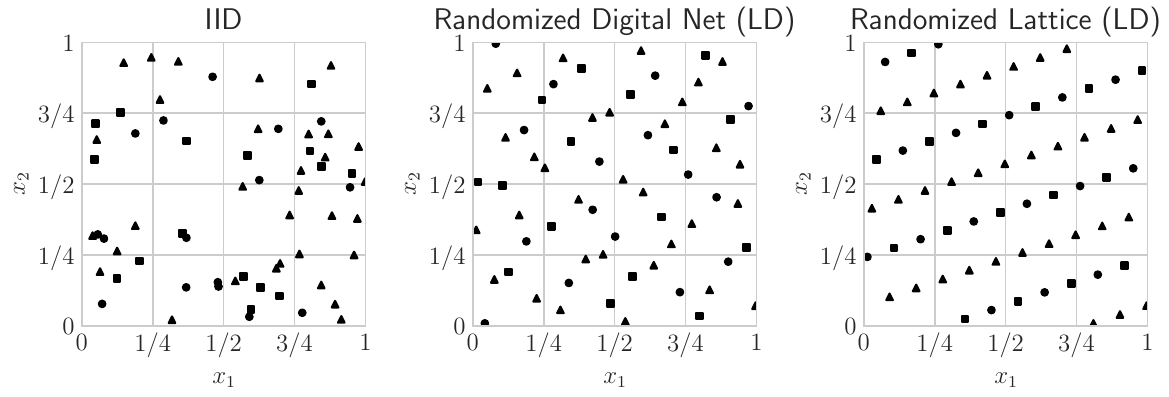}
    \caption{Contrast of IID points with randomized, extensible LD sequences. The first $2^4$ points of each sequence are squares, doubling to the first $2^5$ points adds circles, and doubling again to the first $2^6$ points adds the triangles. Note the gaps and clusters in the IID sequence contrasted with the more even coverage of LD sequences. Also, notice that as the sample size is doubled the extensible LD sequences do not discard previous points but instead fill in the gaps left by previous nodes.}
    \label{SoRa_fig:ld_seqs}
\end{figure}

\section{Bounding a Scalar Mean with MC and QMC}\label{SoRa_sec:Existing_QMC_Methods}

This section discusses some existing (Q)MC methods for inferring bounds $[\mu^-,\mu^+]$ on a scalar mean $\mu$ which hold with uncertainty less than some threshold $\alpha^{(\mu)} \in (0,1)$. Specifically, given nodes $\{\boldsymbol{x}_i\}_{i=1}^n$ and corresponding function evaluations $\{f(\boldsymbol{x}_i)\}_{i=1}^n$, we discuss methods for determining bounds $-\infty \leq \mu^- \leq \mu^+ \leq \infty$ so that $\mu \in [\mu^-,\mu^+]$ with probability greater than or equal to $1-\alpha^{(\mu)}$. Table \ref{SoRa_table:qmcpy_sc} compares the bounding methods discussed in the remainder of this section.

\begin{table}[t]
\centering
\begin{tabular}{r c c c c c c}
    QMCPy Class Name & MC Type & Point Sets & Bounds \\
    \hline
    \texttt{CubMCCLT} \cite{cubmcg} & MC & \texttt{IID} & Probabilistic \\
    \texttt{CubQMCRep} \cite{mcbook} & QMC & \texttt{LD} & Probabilistic \\
    \texttt{CubQMCNetG} \cite{cubqmcsobol} & QMC & \texttt{DigitalNetB2} & Deterministic \\
    \texttt{CubQMCLatticeG} \cite{cubqmclattice} & QMC & \texttt{Lattice} & Deterministic \\
    \texttt{CubQMCBayesNetG} \cite{cubqmcbayessobol} & QMC &  \texttt{DigitalNetB2} & Bayesian \\
    \texttt{CubQMCBayesLatticeG} \cite{cubqmcbayeslattice} & QMC & \texttt{Lattice} & Bayesian \\
    \hline
\end{tabular}
\caption{A comparison of algorithms in the QMCPy library capable of adaptively finding bounds on a scalar mean that hold with high probability. \emph{Type} indicates whether an algorithm is Monte Carlo (MC) or Quasi-Monte Carlo (QMC). \emph{Point Sets} indicate classes of compatible sequences in QMCPy. For example, \texttt{CubQMCRep} is compatible with any low discrepancy (\texttt{LD}) sequence including base 2 digital nets (\texttt{DigitalNetB2}) and integration lattices (\texttt{Lattice}). However, \texttt{CubQMCNetG} is only compatible with \texttt{DigitalNetB2} sequences and will not work with \texttt{Lattice} or other LD sequences. \emph{Bounds} specify the method of error estimation as discussed throughout Section \ref{SoRa_sec:Existing_QMC_Methods}. Deterministic bounds hold with probability $1$ i.e. for any $\alpha^{(\mu)} \in (0,1)$. Probabilistic and Bayesian bounds are tailored to the choice of $\alpha^{(\mu)}$. The GAIL MATLAB library \cite{ChoEtal21a} also implements these algorithms for a scalar mean. }
\label{SoRa_table:qmcpy_sc}
\end{table}

\begin{description}
    \item[\texttt{CubMCCLT}] When $\{\boldsymbol{x}_i\}_{i=1}^n$ are IID and $f$ has a finite variance, the Central Limit Theorem may provide a heuristic $1-\alpha^{(\mu)}$ confidence interval for $\mu$ by setting $\mu^\pm = \hat{\mu} \pm Z_{\alpha^{(\mu)}/2}\sigma/\sqrt{n}$. Here $Z_{\alpha^{(\mu)}/2}$ is the inverse CDF of a standard normal distribution at $1-\alpha^{(\mu)}/2$, and $\hat{\mu}$ is the sample average of function evaluations as in \eqref{SoRa_eq:mcapprox}. The variance of $f(\boldsymbol{X})$ is the generally unknown quantity $\sigma^2$ which may be approximated by the unbiased estimator $\hat{\sigma}^2 = 1/(n-1)\sum_{i=1}^n (f(\boldsymbol{x}_i)-\hat{\mu})^2$, perhaps multiplied by an inflation factor $C^2>1$ for a more conservative estimate. The resulting  heuristic bounds on $\mu$ are $\mu^\pm = \hat{\mu} \pm CZ_{\alpha^{(\mu)}/2} \hat{\sigma} / \sqrt{n}$.
    \item[\texttt{CubMCG}] In \cite{cubmcg}, Hickernell and collaborators extend \texttt{CubMCCLT} to accommodate  finite $n$ and provide bounds that are guaranteed to satisfy the uncertainty threshold. Their two-step method relies on the Berry-Esseen inequality and the assumption that $f$ lies in a cone of functions with known and bounded kurtosis. This method is not readily compatible with the adaptive sampling scheme in Algorithm \ref{SoRa_algo:MCStoppingCriterion}.
    \item[\texttt{CubQMCRep}] This method utilizes IID randomizations of a LD sequence and then derives bounds based on the IID sample averages. Specifically, suppose $\{\boldsymbol{x}_i^{(1)}\}_{i=1}^n,\dots,\{\boldsymbol{x}_i^{(R)}\}_{i=1}^n$ are $R$ IID randomizations of an LD point set. Then one may compute the $R$ IID sample averages $\hat{\mu}_r = \frac{1}{n} \sum_{i=1}^n f(\boldsymbol{x}_i^{(r)})$ for $r = 1,\dots,R$. Similar to what was done for \texttt{CubMCCLT}, one may then compute $\hat{\mu} = 1/R \sum_{r=1}^R \hat{\mu}_r$ and $\hat{\sigma}_R = \sqrt{1/(R-1)\sum_{i=1}^R(\hat{\mu}_r - \hat{\mu})^2}$ to produce heuristic bounds $\mu^\pm = \hat{\mu} \pm C T_{\alpha^{(\mu)}/2,R-1} \hat{\sigma}_R / \sqrt{R}$. Here $C>1$ is still an inflation factor and we now use $T_{\alpha^{(\mu)}/2,R-1}$, the inverse CDF of Student's-$t$ distribution with $R-1$ degrees of freedom, instead of $Z_{\alpha^{(\mu)}/2}$ since $R$ may be small.
    A more careful treatment of this heuristic method is available in \cite[Chapter 17]{mcbook}. 
    \item[\texttt{CubQMC\{Net,Lattice\}G}] Hickernell and Rugama developed algorithms in \cite{adaptive_qmc} that track the decay of Fourier coefficients based on a single randomized LD sequence i.e. $R=1$. These algorithms provide deterministic bounds on $\mu$ for functions in a cone parameterized by the decay rate of the Walsh coefficients for digital sequences \cite{cubqmcsobol} or the complex exponential Fourier coefficients for integration lattices \cite{cubqmclattice}. 
    \item[\texttt{CubQMCBayes\{Net,Lattice\}G}] Another pair of QMC algorithms take a Bayesian approach to error estimation, again using only a single randomized LD sequence. These algorithms assume the integrand is a realization of a Gaussian process. Utilizing special kernels matched to LD sequences enables the Gaussian process to be fit at $\mathcal{O}(n \log n)$ cost. Similar to \texttt{CubQMC\{Net,Lattice\}G}, Fourier and Walsh coefficients are used respectively with lattice and digital nets, except here the coefficients are used to derive credible intervals. Thus it provides a much stronger theoretical background. These algorithms must estimate shape and scale parameters for the covariance kernels which leads to greater computational cost. These Bayesian QMC algorithms \cite{cubqmcbayes_thesis} are also available for both digital nets \cite{cubqmcbayessobol} and integration lattices  \cite{cubqmcbayeslattice}. 
\end{description}


\section{Bounds on a Scalar QOI} \label{SoRa_sec:comb_sol_approx}

This section discusses how to compute bounds $[s^-,s^+]$ on scalar QOI $s = C(\boldsymbol{\mu})$ so that 
\begin{equation}
    P(s \in [s^-,s^+]) \geq 1-\alpha^{(s)}
    \label{SoRa_eq:scalar_comb_desired_ineq}
\end{equation}
where $\alpha^{(s)} \in (0,1)$ is an uncertainty threshold on the QOI bounds. Here $C: \mathbb{R}^{\boldsymbol{d}_{\boldsymbol{\mu}}} \to \mathbb{R}$ combines the array mean $\boldsymbol{\mu}$ into a scalar QOI $s$. First, we discuss how to set the array of mean uncertainty thresholds $\boldsymbol{\alpha}^{(\boldsymbol{\mu})} \in (0,1)^{\boldsymbol{d}_{\boldsymbol{\mu}}}$ so the resulting mean bounds $[\boldsymbol{\mu}^-,\boldsymbol{\mu}^+]$ contain $\boldsymbol{\mu}$ with uncertainty below $\alpha^{(s)}$. Then we discuss how the user may utilize $C$ to define \emph{bound functions} $C^-,C^+: \mathbb{R}^{\boldsymbol{d}_{\boldsymbol{\mu}}} \times \mathbb{R}^{\boldsymbol{d}_{\boldsymbol{\mu}}} \to \mathbb{R}$ so that setting $s^- = C^-(\boldsymbol{\mu}^-,\boldsymbol{\mu}^+)$ and $s^+ = C^+(\boldsymbol{\mu}^-,\boldsymbol{\mu}^+)$ ensures \eqref{SoRa_eq:scalar_comb_desired_ineq} is satisfied.

Let $N = \lvert \boldsymbol{d}_{\boldsymbol{\mu}} \rvert$ be the number of elements in a $\mathbb{R}^{\boldsymbol{d}_{\boldsymbol{\mu}}}$ array and set each element of $\boldsymbol{\alpha}^{(\boldsymbol{\mu})}$ to the constant $\alpha^{(s)}/ N$. Then Boole's inequality \cite{boole1847mathematical} implies that if $[\boldsymbol{\mu}^-,\boldsymbol{\mu}^+]$ are chosen so that
\begin{equation}
    P(\mu_{\boldsymbol{k}} \in [\mu_{\boldsymbol{k}}^-,\mu_{\boldsymbol{k}}^+]) \geq 1-\alpha_{\boldsymbol{k}}^{(\boldsymbol{\mu})}  \qquad \text{for all }\boldsymbol{1} \leq \boldsymbol{k} \leq \boldsymbol{d}_{\boldsymbol{\mu}},
    \label{SoRa_eq:indv_prob_bounds}
\end{equation}
then  
\begin{equation}
    P(\boldsymbol{\mu} \in [\boldsymbol{\mu}^-,\boldsymbol{\mu}^+]) \geq 1-\alpha^{(s)}.
    \label{SoRa_eq:indv_prob_bounds_all}
\end{equation}
The bounds in \eqref{SoRa_eq:indv_prob_bounds} may be found using the methods in Section \ref{SoRa_sec:Existing_QMC_Methods}.

To propagate bounds $[\boldsymbol{\mu}^-,\boldsymbol{\mu}^+]$ on mean $\boldsymbol{\mu}$ to bounds $[s^-,s^+]$ on QOI $s$, the user must define functions $C^-$ and $C^+$ using interval arithmetic \cite{interval_analysis} and problem specific knowledge. These functions must ensure $s \in [C^-(\boldsymbol{\mu}^-,\boldsymbol{\mu}^+),C^+(\boldsymbol{\mu}^-,\boldsymbol{\mu}^+)]$ whenever $\boldsymbol{\mu} \in [\boldsymbol{\mu}^-,\boldsymbol{\mu}^+]$. Without problem specific knowledge, one may set 
\begin{equation}
    s^- = C^-(\boldsymbol{\mu}^-,\boldsymbol{\mu}^+) = \min_{\boldsymbol{\mu} \in [\boldsymbol{\mu}^-,\boldsymbol{\mu}^+]} C(\boldsymbol{\mu}), \quad 
    s^+= C^+(\boldsymbol{\mu}^-,\boldsymbol{\mu}^+) = \max_{\boldsymbol{\mu} \in [\boldsymbol{\mu}^-,\boldsymbol{\mu}^+]} C(\boldsymbol{\mu}).
    \label{SoRa_eq:C_minus_C_plus}
\end{equation}
Table \ref{SoRa_table:elementary_ops_Cpm} provides examples of such interval arithmetic functions for some basic operations. Problem specific knowledge may be used to further shrink the naive bounds in bounds \eqref{SoRa_eq:C_minus_C_plus}. For example, if $s$ is a probability then $0 \leq s^- \leq s^+ \leq 1$ may be encoded into $C^-$ and $C^+$. See Section \ref{SoRa_sec:sensitivity_indices} for a more nuanced example. 

In all, given $\alpha^{(s)}$, we may set $\boldsymbol{\alpha}^{(\boldsymbol{\mu})}=\alpha^{(s)}/N$ elementwise then use scalar (Q)MC  algorithms to find $[\boldsymbol{\mu}^-,\boldsymbol{\mu}^+]$ satisfying \eqref{SoRa_eq:indv_prob_bounds} so that \eqref{SoRa_eq:indv_prob_bounds_all} holds. Then setting $[s^-,s^+]$ via \eqref{SoRa_eq:C_minus_C_plus}, potentially combined with problem specific knowledge, guarantees \eqref{SoRa_eq:scalar_comb_desired_ineq} holds.

\begin{table}[t]
\begin{tabular}{r  c  c}
    $s=C(\boldsymbol{\mu})$ & $s^- = C^-(\boldsymbol{\mu}^-,\boldsymbol{\mu}^+)$ & $s^+ = C^+(\boldsymbol{\mu}^-,\boldsymbol{\mu}^+)$ \\
    \hline
    $\mu_1+\mu_2$ & $\mu_1^-+\mu_2^-$ & $\mu_1^++\mu_2^+$ \\
    $\mu_1-\mu_2$ & $\mu_1^--\mu_2^+$ & $\mu_1^+-\mu_2^-$ \\
    $\mu_1 \cdot \mu_2$ & $\min(\mu_1^-\mu_2^-,\mu_1^-\mu_2^+,\mu_1^+\mu_2^-,\mu_1^+\mu_2^+)$ & $\max(\mu_1^-\mu_2^-,\mu_1^-\mu_2^+,\mu_1^+\mu_2^-,\mu_1^+\mu_2^+)$ \\
    $\mu_1 / \mu_2$ & $\begin{cases} -\infty, & 0 \in [\mu_2^-,\mu_2^+] \\ \min\left(\frac{\mu_1^-}{\mu_2^-},\frac{\mu_1^+}{\mu_2^-},\frac{\mu_1^-}{\mu_2^+},\frac{\mu_1^+}{\mu_2^+}\right), & 0 \notin [\mu_2^-,\mu_2^+] \end{cases}$ & $\begin{cases} \infty, & 0 \in [\mu_2^-,\mu_2^+] \\ \max\left(\frac{\mu_1^-}{\mu_2^-},\frac{\mu_1^+}{\mu_2^-},\frac{\mu_1^-}{\mu_2^+},\frac{\mu_1^+}{\mu_2^+}\right), & 0 \notin [\mu_2^-,\mu_2^+] \end{cases}$ \\
    $\min(\mu_1,\mu_2)$ & $\min(\mu_1^-,\mu_2^-)$ & $\min(\mu_1^+,\mu_2^+)$ \\
    $\max(\mu_1,\mu_2)$ & $\max(\mu_1^-,\mu_2^-)$ & $\max(\mu_1^+,\mu_2^+)$ \\
    \hline
\end{tabular}
\caption{Interval arithmetic functions for elementary operations.}
\label{SoRa_table:elementary_ops_Cpm}
\end{table}

\section{Optimal Approximation of a Scalar QOI} \label{SoRa_sec:opt_comb_sol_sc}

This section derives a stopping criterion for adaptive sampling and optimal approximation $\hat{s}$ of scalar QOI $s$. Let $h^{(\varepsilon)}: \mathbb{R} \to \mathbb{R}^+$ be an error metric dependent on some error tolerance $\varepsilon$ so that the stopping criterion is met if and only if the QOI approximation $\hat{s}$ satisfies 
\begin{equation}
    \lvert s-\hat{s} \rvert \leq h^{(\varepsilon)}(s) \quad \forall s \in [s^-,s^+].
    \label{SoRa_eq:sc_raw}
\end{equation}
Theorem \ref{SoRa_thm:shat_opt} determines the optimal $\hat{s}$ and an equivalent condition to \eqref{SoRa_eq:sc_raw} when $h^{(\varepsilon)}$ is a metric map i.e. Lipschitz continuous with constant at most $1$. Some compatible error metric options are
\begin{subequations}
\begin{align}
    h^{(\varepsilon)}(s) &= \max\left(\varepsilon^\text{abs},\lvert s \rvert \varepsilon^\text{rel} \right) \quad &&\text{absolute or relative error satisfied,} \label{SoRa_eq:h_abs_or_rel} \\
    h^{(\varepsilon)}(s) &= \min\left(\varepsilon^\text{abs},\lvert s \rvert \varepsilon^\text{rel} \right) \quad &&\text{absolute and relative error satisfied.} \label{SoRa_eq:h_abs_and_rel}
\end{align}
\end{subequations}
\begin{theorem} \label{SoRa_thm:shat_opt}
    Suppose that  $h^{(\varepsilon)}$ satisfies the metric map condition
    \begin{equation}
        \lvert h^{(\varepsilon)}(s_1) - h^{(\varepsilon)}(s_2) \rvert \leq \lvert s_1 - s_2 \rvert \qquad \text{for all } s_1,s_2 \in \mathbb{R}.
        \label{SoRa_eq:metric_map_cond}
    \end{equation}
    Then error criterion  \eqref{SoRa_eq:sc_raw} holds if and only if 
    \begin{equation}
        s^+-s^- \leq h^{(\varepsilon)}(s^-)+h^{(\varepsilon)}(s^+).
        \label{SoRa_eq:sc}
    \end{equation}
    Furthermore, the choice of 
    \begin{equation}
        \hat{s} = \frac{1}{2}\left[s^-+s^++h^{(\varepsilon)}(s^-)-h^{(\varepsilon)}(s^+)\right]
        \label{SoRa_eq:shat_opt}
    \end{equation}
    minimizes $\sup_{s \in [s^-,s^+]} \lvert s - \hat{s} \rvert -h^{(\varepsilon)}(s)$ for any choice of $s^{\pm}$ with $s^- < s^+$.
\end{theorem}

\begin{proof}
    Define $g(s,\hat{s})=\lvert s - \hat{s} \rvert -h^{(\varepsilon)}(s)$. From \eqref{SoRa_eq:metric_map_cond}, it follows that if  $s^- \leq s \leq \hat{s}$ then $g(s^-,\hat{s})-g(s,\hat{s}) \geq 0$, and if $\hat{s} \leq s \leq s^+$ then $g(s^+,\hat{s})-g(s,\hat{s})  \geq 0$. This means that $g(\cdot,\hat{s})$ attains its maximum at either $s^-$ or $s^+$ so that
    \begin{equation*}
        \max_{s \in [s^-,s^+]} g(s,\hat{s}) = \max_{s \in \{s^-,s^+\}} g(s,\hat{s}).
    \end{equation*}
    
    Next, we find the optimal choice of $\hat{s}$.  The function $g(s^-,\cdot)$ is monotonically decreasing to the left of  $s^-$ and monotonically increasing to the right of $s^-$. Similarly, $g(s^+,\cdot)$ is monotonically decreasing to the left of $s^+$ and monotonically increasing to the right of $s^+$. This means that the optimal choice of $\hat{s}$ to minimize $\max_{s \in \{s^-,s^+\}} g(s,\hat{s})$ lies in $[s^-,s^+]$ and satisfies $g(s^-,\hat{s}) = g(s^+,\hat{s})$, that is, 
    $$\hat{s} - s^- - h^{(\varepsilon)}(s^-) = s^+ - \hat{s} - h^{(\varepsilon)}(s^+).$$
    Solving for the optimal value of $\hat{s}$ leads to \eqref{SoRa_eq:shat_opt}.
    
    For this optimal $\hat{s}$, 
    $$2 \max_{s \in [s^-,s^+]} g(s,\hat{s}) =  s^+  -  s^-  - h^{(\varepsilon)}(s^-) - h^{(\varepsilon)}(s^+).$$
    The error criterion is equivalent to $\max_{s \in [s^-,s^+]} g(s,\hat{s}) \le 0 $.  This can only hold under condition  \eqref{SoRa_eq:sc}. 
\end{proof}

\section{Adaptive Algorithm with Extension to Array QOI} \label{SoRa_sec: Vectorized Implementation}

In the previous section, we assumed an array mean $\boldsymbol{\mu}$ was used to compute a scalar QOI $s$. We now relax these assumptions to enable approximation of array QOI $\boldsymbol{s}$. The optimal approximation $\hat{\boldsymbol{s}}$ and stopping criterion may still be computed by elementwise application of \eqref{SoRa_eq:shat_opt} and \eqref{SoRa_eq:sc}. 

For some integrands $\boldsymbol{f}$ it is possible to avoid evaluating particular integrand outputs when all affected QOI have already been sufficiently approximated. In such cases, the user may enable economic evaluation by defining a dependency function $\boldsymbol{D}: \{\text{True},\text{False}\}^{\boldsymbol{d}_{\boldsymbol{s}}} \to \{\text{True},\text{False}\}^{\boldsymbol{d}_{\boldsymbol{\mu}}}$ which maps stopping flags on QOI to stopping flags on means. The latter indicates which outputs the integrand is required to compute in the next iteration.  We say (QOI) index $\boldsymbol{1} \leq \boldsymbol{l} \leq \boldsymbol{d}_{\boldsymbol{s}}$ depends on (mean) index $\boldsymbol{1} \leq \boldsymbol{k} \leq \boldsymbol{d}_{\boldsymbol{\mu}}$ if the $\boldsymbol{k}^\text{th}$ entry is $\text{True}$ in the output of evaluating $\boldsymbol{D}$ at the multi-dimensional array with only the $\boldsymbol{l}^\text{th}$ entry set to $\text{True}$.

Moreover, $\,\boldsymbol{D}$ may be used to compute mean uncertainty levels $\boldsymbol{\alpha}^{(\boldsymbol{\mu})}$ in the spirit of Boole's inequality as done in Section \ref{SoRa_sec:comb_sol_approx}. The idea is to ensure that each element of $\boldsymbol{\alpha}^{(\boldsymbol{s})}$ is greater than the sum of elements in $\boldsymbol{\alpha}^{(\boldsymbol{\mu})}$ with dependent indices. Specifically, let $\boldsymbol{N} \in \mathbb{N}^{\boldsymbol{d}_{\boldsymbol{s}}}$ contain the number of dependent mean for each QOI. That is, for every $\boldsymbol{1} \leq \boldsymbol{l} \leq \boldsymbol{d}_{\boldsymbol{s}}$, $N_{\boldsymbol{l}}$ is the number of indices dependent on $\boldsymbol{l}$. For every $\boldsymbol{1} \leq \boldsymbol{k} \leq \boldsymbol{d}_{\boldsymbol{\mu}}$, if $\boldsymbol{l}$ is dependent on $\boldsymbol{k}$ then $\alpha_{\boldsymbol{l}}^{(\boldsymbol{s})}/N_{\boldsymbol{l}}$ is a candidate for $\alpha_{\boldsymbol{k}}^{(\boldsymbol{\mu})}$. We then set $\alpha_{\boldsymbol{k}}^{(\boldsymbol{\mu})}$ to the minimum amongst all candidates for $\alpha_{\boldsymbol{k}}^{(\boldsymbol{\mu})}$, assuming the candidate set is not empty.

While not theoretically required, our implementation practically requires that each index $\boldsymbol{1} \leq \boldsymbol{k} \leq \boldsymbol{d}_{\boldsymbol{\mu}}$ be a dependency of exactly one index $\boldsymbol{1} \leq \boldsymbol{l} \leq \boldsymbol{d}_{\boldsymbol{s}}$. To illustrate this requirement and the previously discussed dependency structure, let us consider the simple example with QOI $s_1 = \mu_1 + \mu_2$ and $s_2 = \mu_1 + \mu_3$. Suppose after some iteration that $s_1$ is sufficiently approximated and $s_2$ is not. Since $\mu_1$ is a dependency of $s_2$, we would like to continue sampling for $\mu_1$ to get a better approximation of $s_2$. However, changes in the bounds on $\mu_1$ will change the bounds on $s_1$ and may potentially make the approximation of $s_1$ become insufficient again. This out of sync nature of the sampling occurs because $\mu_1$ is a dependency of more than one QOI. To remedy this, let $\mu_4 = \mu_1$ and set $s_2 = \mu_4 + \mu_3$. Now each mean is a dependency of exactly one QOI as required. In practice this remedy amounts to copying integrand outputs at index $1$ into index $4$, thus increasing storage requirements in favor of potentially avoiding evaluating integrand outputs at index $2$ or $3$. This dependency structure is illustrated below with dependency function $D(b_1,b_2) = (b_1,b_1,b_2,b_2)$ where $b_1,b_2 \in \{\text{True},\text{False}\}$. 

\begin{figure}
    \centering
\begin{tikzpicture}[main/.style = {draw, circle}] 
    \node at (0,0)  [main,line width=3pt] (1) {$\mu_1$}; 
    \node at (2,0)  [main,line width=3pt] (2) {$\mu_2$};
    \node at (4,0)  [main,line width=3pt] (3) {$\mu_3$};
    \node at (6,0)  [main,line width=3pt] (4) {$\mu_4$};
    \node at (1,1) [main,line width=3pt] (5) {$s_1$};
    \node at (5,1) [main,line width=3pt] (6) {$s_2$};
    \draw[->,line width=3pt] (5) -- (1);
    \draw[->,line width=3pt] (5) -- (2);
    \draw[->,line width=3pt] (6) -- (3);
    \draw[->,line width=3pt] (6) -- (4);
\end{tikzpicture}
\end{figure}

Algorithm \ref{SoRa_algo:MCStoppingCriterion} details the adaptive procedure developed throughout this article. Notice that the implementation does not require specifying $\boldsymbol{C}$ despite its use in deriving the necessary $\boldsymbol{C}^-$ and $\boldsymbol{C}^+$ inputs. The cost of this algorithm is concentrated on evaluating the function at an IID or LD sequence. In practice, the run time may be reduced through parallel and/or economic evaluation.

\begin{algorithm}[t]
    \caption{Adaptive (Quasi-)Monte Carlo for Array QOI}
    \label{SoRa_algo:MCStoppingCriterion}
    \begin{algorithmic}
    \Require $\boldsymbol{f}: (0,1)^d \to \mathbb{R}^{\boldsymbol{d}_{\boldsymbol{\mu}}}$, the integrand where $\boldsymbol{\mu} = \mathbb{E}[\boldsymbol{f}(\boldsymbol{X})]$ for $\boldsymbol{X} \sim \mathcal{U}[0,1]^d$.
    \Require $\boldsymbol{C}^-,\boldsymbol{C}^+: \mathbb{R}^{\boldsymbol{d}_{\boldsymbol{\mu}}} \times \mathbb{R}^{\boldsymbol{d}_{\boldsymbol{\mu}}} \to \mathbb{R}^{\boldsymbol{d}_{\boldsymbol{s}}}$, generalization of \eqref{SoRa_eq:C_minus_C_plus} so $\boldsymbol{\mu} \in [\boldsymbol{\mu}^-,\boldsymbol{\mu}^+]$ implies $\boldsymbol{s} \in [\boldsymbol{C}^-(\boldsymbol{\mu}^-,\boldsymbol{\mu}^+),\boldsymbol{C}^+(\boldsymbol{\mu}^-,\boldsymbol{\mu}^+)]$.
    \Require $\boldsymbol{\alpha}^{(\boldsymbol{s})} \in (0,1)^{\boldsymbol{d}_{\boldsymbol{s}}}$, the desired uncertainty thresholds on QOI bounds so the returned $[\boldsymbol{s}^-,\boldsymbol{s}^+]$ will satisfy $P(s_{\boldsymbol{l}} \in [s_{\boldsymbol{l}}^-,s_{\boldsymbol{l}}^+]) \geq 1-\alpha^{(\boldsymbol{s})}_{\boldsymbol{l}}$ for any $\boldsymbol{1} \leq \boldsymbol{l} \leq \boldsymbol{d}_{\boldsymbol{s}}$.
    \Require $h^{(\varepsilon_{\boldsymbol{l}})}_{\boldsymbol{l}}: \mathbb{R} \to \mathbb{R}^+$ for $\boldsymbol{1} \leq \boldsymbol{l} \leq \boldsymbol{d}_{\boldsymbol{l}}$, see \eqref{SoRa_eq:h_abs_or_rel} or \eqref{SoRa_eq:h_abs_and_rel} for examples. Stopping flag at index $\boldsymbol{l}$ is set to $\text{True}$ when $\lvert s_{\boldsymbol{l}} - \hat{s}_{\boldsymbol{l}} \rvert \leq h^{(\varepsilon_{\boldsymbol{l}})}_{\boldsymbol{l}}(s_{\boldsymbol{l}})$ for all $s_{\boldsymbol{l}} \in [s_{\boldsymbol{l}}^-,s_{\boldsymbol{l}}^+]$.
    \Require $\boldsymbol{D}: \{\text{True},\text{False}\}^{\boldsymbol{d}_{\boldsymbol{s}}} \to \{\text{True},\text{False}\}^{\boldsymbol{d}_{\boldsymbol{\mu}}}$, maps stopping flags on QOI $\boldsymbol{s}$ to stopping flags on mean $\boldsymbol{\mu}$. 
    \Require A scalar (Q)MC algorithm capable of producing bounds on a mean which holds with uncertainty below a specified threshold. See the methods in Section \ref{SoRa_sec:Existing_QMC_Methods}.
    \Require A generator of IID or LD sequences compatible with the scalar MC or QMC algorithm. 
    \Require $m_1 \in \mathbb{N}$, where $2^{m_1}$ is the initial number of samples.
    \\ \hrulefill
    \State $n_\text{start} \gets 1$ \Comment{lower index in node sequence(s)}
    \State $n_\text{end} \gets 2^{m_1}$ \Comment{upper index in node sequence(s)}
    \State $\boldsymbol{b}^{(\boldsymbol{\mu})} \gets \text{False}^{\boldsymbol{d}_{\boldsymbol{\mu}}}$ \Comment{stopping flags on the mean}
    \State $\boldsymbol{b}^{(\boldsymbol{s})} \gets \text{False}^{\boldsymbol{d}_{\boldsymbol{s}}}$ \Comment{stopping flags on QOI}
    \State Set $\boldsymbol{\alpha}^{(\boldsymbol{\mu})}$ based on $\boldsymbol{\alpha}^{(\boldsymbol{s})}$ using $\boldsymbol{D}$ \Comment{See discussion in Section \ref{SoRa_sec: Vectorized Implementation}}
    \While{$b_{\boldsymbol{l}}^{(\boldsymbol{s})} = \text{False}$ for some $\boldsymbol{1} \leq \boldsymbol{l} \leq \boldsymbol{d}_{\boldsymbol{s}}$} \Comment{A QOI element is not sufficiently bounded}
        \State Generate nodes from the IID or LD sequence(s) from index $n_\text{start}$ to $n_\text{end}$
        \State Evaluate $\boldsymbol{f}$ at the new nodes where $\boldsymbol{b}^{(\boldsymbol{\mu})}=\textbf{\text{False}}$ \Comment{may be done in parallel}
        \State Update $\boldsymbol{\mu}^-$ and $\boldsymbol{\mu}^+$ using the scalar (Q)MC algorithm where $\boldsymbol{b}^{(\boldsymbol{\mu})}=\textbf{\text{False}}$
        \State $[\boldsymbol{s}^-,\boldsymbol{s}^+] \gets \left[\boldsymbol{C}^-(\boldsymbol{\mu}^-,\boldsymbol{\mu}^+),\boldsymbol{C}^+(\boldsymbol{\mu}^-,\boldsymbol{\mu}^+)\right]$ 
        \State $b^{(\boldsymbol{s})}_{\boldsymbol{l}} \gets \text{Boolean}\left(s_{\boldsymbol{l}}^+-s_{\boldsymbol{l}}^- < h^{(\varepsilon_{\boldsymbol{l}})}_{\boldsymbol{l}}(s_{\boldsymbol{l}}^-)+h^{(\varepsilon_{\boldsymbol{l}})}_{\boldsymbol{l}}(s_{\boldsymbol{l}}^+)\right),\quad \boldsymbol{1} \leq \boldsymbol{l} \leq \boldsymbol{d}_{\boldsymbol{s}}$ \Comment{\eqref{SoRa_eq:sc} elementwise}
        \State $\boldsymbol{b}^{(\boldsymbol{\mu})} \gets \boldsymbol{D}\left(\boldsymbol{b}^{(\boldsymbol{s})}\right)$
        \State $n_\text{start} \gets n_\text{end}+1$
        \State $n_\text{end} \gets 2n_\text{start}$
    \EndWhile
    \State $\hat{s}_{\boldsymbol{l}} \gets \frac{1}{2}[s_{\boldsymbol{l}}^-+s_{\boldsymbol{l}}^++h^{(\varepsilon_{\boldsymbol{l}})}_{\boldsymbol{l}}(s_{\boldsymbol{l}}^-)-h^{(\varepsilon_{\boldsymbol{l}})}_{\boldsymbol{l}}(s_{\boldsymbol{l}}^+)], \quad\boldsymbol{1} \leq \boldsymbol{l} \leq \boldsymbol{d}_{\boldsymbol{s}}$ \Comment{\eqref{SoRa_eq:shat_opt} elementwise}
    \State \Return $\hat{\boldsymbol{s}},[\boldsymbol{s}^-,\boldsymbol{s}^+]$
    \end{algorithmic}
\end{algorithm}

\section{Examples} \label{SoRa_sec:examples}

This section presents a number of examples spanning machine learning and sensitivity analysis. Code implementing these examples in the QMCPy framework and reproducing the figures in this article is available in \cite{vectorized_qmc_demo_notebook}. More details on the QMCPy framework are available in \cite{QMCSoftware}.

\subsection{Vectorized Acquisition Functions for Bayesian Optimization}

Bayesian optimization (BO) is a sequential optimization technique that attempts to find the global maximum of a black box function $\varphi: (0,1)^{\nu} \to \mathbb{R}$. It is assumed that $\varphi$ is expensive to evaluate, so we must strategically select sampling locations that maximize some utility or acquisition function. At a high level, BO 
\begin{enumerate}
    \item Iteratively samples $\varphi$ at locations maximizing the acquisition function,
    \item Updates a Gaussian process surrogate based on these new observations,
    \item Updates an acquisition function based on the updated surrogate,
    \item repeats until the budget for sampling $\varphi$ has expired.
\end{enumerate}
Bayesian optimization is detailed in \cite{snoek2012practical} while Gaussian process regression is given individual treatment in \cite{rasmussen2006gaussian}.

Concretely, suppose we have already sampled $\varphi$ at $\boldsymbol{z}_1,\dots,\boldsymbol{z}_{N} \in [0,1]^{\nu}$ to collect data $\mathcal{D}=\{(\boldsymbol{z}_i,y_i)\}_{i=1}^N$ where $y_i = \varphi(\boldsymbol{z}_i)$. BO may then fit a Gaussian process surrogate to data $\mathcal{D}$. The next $d$ sampling locations may then be chosen to maximize an acquisition function $\alpha: [0,1]^{(d,\nu)} \to \mathbb{R}$ which takes a matrix whose rows are the next sampling locations to a payoff value. Specifically, we set $\boldsymbol{z}_{N+1}, \dots, \boldsymbol{z}_{N+d}$ to be the rows of $\mathop{\text{argmax}}_{\boldsymbol{Z} \in [0,1]^{(d,\nu)}}\alpha(\boldsymbol{Z})$. Many acquisition functions may be expressed as an expectation of the form $\alpha(\boldsymbol{Z}) = \mathbb{E}\left[a(\boldsymbol{y}) \mid \boldsymbol{y} \sim \mathcal{N}\left(\boldsymbol{m},\boldsymbol{\Sigma}\right)\right]$ where $\boldsymbol{m} \in \mathbb{R}^{d}$ and  $\boldsymbol{\Sigma} \in \mathbb{R}^{(d,d)}$ are respectively the posterior mean and covariance of the Gaussian process at points $\boldsymbol{Z}$. Here we focus on the q-Expected Improvement (qEI) acquisition function which uses $a(\boldsymbol{y}) = \max_{1 \leq i \leq d} (y_i - y^*)_+$ where $y^*= \max\left(y_1,\dots,y_N\right)$ is the current maximum and $(\cdot)_+ = \max(\cdot,0)$. 

Suppose we choose the argument maximum  from among a finite set of $d$-sized batches $\boldsymbol{Z}_1,\dots,\boldsymbol{Z}_{d_{\boldsymbol{\mu}}} \in [0,1]^{(d,\nu)}$ so that $\boldsymbol{z}_{N+1}, \dots,\boldsymbol{z}_{N+d}$ are set to be the rows of $\mathop{\text{argmax}}_{\boldsymbol{Z} \in \{\boldsymbol{Z}_1,\dots,\boldsymbol{Z}_{d_{\boldsymbol{\mu}}}\}}\alpha(\boldsymbol{Z})$. We may vectorize the acquisition function computations to $s_i = \mu_i = \alpha(\boldsymbol{Z}_i) = \mathbb{E}\left[a\left(\boldsymbol{A}_i\boldsymbol{\Phi}^{-1}(\boldsymbol{X})+\boldsymbol{m}_i\right)\right]$ for $i=1,\dots,d_\mu$ where $\boldsymbol{X} \sim \mathcal{U}(0,1)^d$ and $\boldsymbol{\Phi}^{-1}$ is the inverse CDF of the standard Gaussian taken elementwise. Now $\boldsymbol{m}_i$ and $\boldsymbol{\Sigma}_i = \boldsymbol{A}_i\boldsymbol{A}_i^T$ are the posterior mean and covariance respectively of the Gaussian process at $\boldsymbol{Z}_i$ so that $\boldsymbol{A}_i\boldsymbol{\Phi}^{-1}(\boldsymbol{X})+\boldsymbol{m}_i \sim \mathcal{N}\left(\boldsymbol{m}_i,\boldsymbol{\Sigma}_i\right)$ for $i=1,\dots,d_{\boldsymbol{\mu}}$.

Since the quantity of interest is simply the vector of expectations, one may set $\boldsymbol{C}$, $\boldsymbol{C}^-$, $\boldsymbol{C}^+$, and $\boldsymbol{D}$ to appropriate identity functions. The process described above is visualized in Figure \ref{SoRa_fig:bo_qei} for $\nu=1$ and $d=2$. In this example, it may be more intuitive to make $\boldsymbol{d}_{\boldsymbol{\mu}}$ have length $2$ so the matrix of means reflects the grid of white dots in the right panel of Figure \ref{SoRa_fig:bo_qei}. 


\begin{figure}[t]
    \centering
    \includegraphics[width=.9\textwidth]{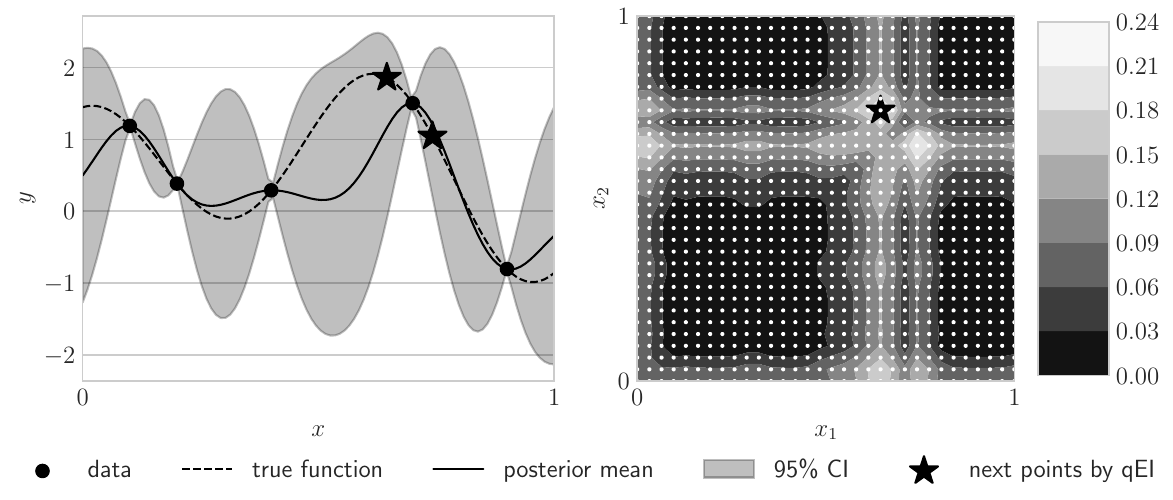}
    \caption{First, the true function has been sampled at the data points shown in the left panel. Next, a Gaussian Process is fit to the data points to approximate the true function. The posterior mean and $95\%$ confidence interval (CI) of the Gaussian Process are shown in the left panel. With $d=2$, a fine grid of candidates is chosen in $[0,1]^{2}$ and depicted in the right panel. The presented (Q)MC algorithm is then used to approximate the acquisition function value at each of the candidate grid points. These approximations are made into a contour plot in the right panel. The discrete argument maximum among these approximations on the fine grid is the next size $d$ batch of points by qEI. These next points for sequential optimization are visualized in both the right and left panels.}
    \label{SoRa_fig:bo_qei}
\end{figure}

\subsection{Bayesian Posterior Mean}

The Bayesian framework combines prior knowledge of random parameters $\boldsymbol{\Theta} \in \mathbb{R}^{d_{\boldsymbol{s}}}$ with observational data and a likelihood function $\rho$ to construct a model-aware posterior distribution on $\boldsymbol{\Theta}$. Suppose we have a dataset of observations $\boldsymbol{y} = (y_1,\dots,y_{N})$ taken at IID locations $\boldsymbol{z}_1,\dots,\boldsymbol{z}_{N}$ respectively. Then Bayes' rule may be used to write the posterior density of $\boldsymbol{\Theta}$ as 
$$P\left(\boldsymbol{\theta} \mid \boldsymbol{y} \right) = \frac{P(\boldsymbol{y} \mid \boldsymbol{\theta}) P(\boldsymbol{\theta})}{P\left(\boldsymbol{y}\right)} = \frac{\prod_{i=1}^{N} \rho(y_i \mid \boldsymbol{\theta}) P(\boldsymbol{\theta})}{\mathbb{E}\left[\prod_{i=1}^{N} \rho(y_i \mid \boldsymbol{\theta})\right]}.$$
Here the expectation is taken with respect to the prior distribution on $\boldsymbol{\Theta}$ with density $P(\boldsymbol{\theta})$, and $P\left(\boldsymbol{y} \mid \boldsymbol{\theta} \right)$ is the likelihood density which factors into the product of likelihoods $\rho(y_i \mid \boldsymbol{\theta})$ since the observations are IID. 

A useful quantity of interest is the posterior mean of $\boldsymbol{\Theta}$. In this example, the QOI is posterior mean $\boldsymbol{s}$ which may be written as the ratio of expectations via $\boldsymbol{s} = \mathbb{E}\left[\boldsymbol{\Theta} \mid \boldsymbol{y}\right] = \mathbb{E}\left[\boldsymbol{\Theta} \; \prod_{i=1}^{N} \rho(y_i \mid \boldsymbol{\Theta})\right]/\mathbb{E}\left[\prod_{i=1}^{N} \rho(y_i \mid \boldsymbol{\Theta})\right]$. As before, the expectations are taken with respect to the prior distribution on $\boldsymbol{\Theta}$. In the framework of this article $\boldsymbol{\mu} \in \mathbb{R}^{(2, d_{\boldsymbol{s}})}$ where for $k=1,\dots,d_{\boldsymbol{s}}$ we have 
$$\mu_{0k} = \mathbb{E}\left[\Theta_k \prod_{i=1}^{N} \rho(y_i \mid \boldsymbol{\Theta})\right], \quad \mu_{1k} = \mathbb{E}\left[\prod_{i=1}^{N} \rho(y_i \mid \boldsymbol{\Theta})\right], \quad \text{and} \quad s_k = \frac{\mu_{0k}}{\mu_{1k}}.$$
Defining $\boldsymbol{C}^-$ and $\boldsymbol{C}^+$ follow from vectorizing the quotient forms in Table \ref{SoRa_table:elementary_ops_Cpm} while the dependency function $\boldsymbol{D}: \{\text{True},\text{False}\}^{d_{\boldsymbol{s}}} \to \{\text{True},\text{False}\}^{(2, d_{\boldsymbol{s}})}$ is defined by stacking the row vectors of QOI flags on top of itself.

\subsection{Sensitivity Indices}\label{SoRa_sec:sensitivity_indices}

Sensitivity analysis quantifies how uncertainty in a function output may be attributed to subsets of function inputs. Functional ANOVA (analysis of variance) decomposes a function $\varphi \in L^2(0,1)^\nu$ into the sum of orthogonal functions $(\varphi_u)_{u \subseteq {1:\nu}}$. Here $1:\nu=\{1,\dots,\nu\}$ denotes the set of all dimensions and $\varphi_u \in L^2(0,1)^{\lvert u \rvert}$ denotes a sub-function dependent only on inputs $\boldsymbol{x}_u = (x_j)_{j \in u}$ where $\lvert u \rvert$ is the cardinality of $u$. By construction, these sub-functions sum to the objective function so that $\varphi(\boldsymbol{x}) = \sum_{u \subseteq 1:\nu} \varphi_u(\boldsymbol{x}_u)$ \cite[Appendix A]{mcbook}. 
The orthogonality of sub-functions enables the variance of $\varphi$ to be decomposed into the sum of variances of sub-functions. Specifically, denoting the variance of $\varphi$ by $\sigma^2$, we may write $\sigma^2 = \sum_{u \subseteq 1:\nu} \sigma^2_u$ where $\sigma^2_u$ is the variance of sub-function $\varphi_u$. The sub-variance $\sigma_u$ quantifies the variance of $\varphi$ attributable to inputs $u \subseteq 1:\nu$.  The \emph{closed and total Sobol' indices}
\begin{align*}
    \underline{\tau}_u^2 &= \sum_{v \subseteq u} \sigma^2_v = \int_{[0,1]^{2\nu}} f(\boldsymbol{x})[f(\boldsymbol{x}_{u_j},\boldsymbol{z}_{-{u_j}})-f(\boldsymbol{z})]\mathrm{d}\boldsymbol{x}\mathrm{d}\boldsymbol{z} \quad \text{and}  \\ 
    \overline{\tau}_u^2 &= \sum_{v \cap u \neq \emptyset} \sigma^2_v = \frac{1}{2}\int_{[0,1]^{2\nu}} [f(\boldsymbol{z})-f(\boldsymbol{x}_u,\boldsymbol{z}_{-{u_j}})]^2\mathrm{d}\boldsymbol{x}\mathrm{d}\boldsymbol{z}
    \label{SoRa_eq:sobol_indices}
\end{align*}
quantify the variance attributable to subsets of $u$ and subsets containing $u$ respectively. Here the notation $(\boldsymbol{x}_{u},\boldsymbol{z}_{-u})$ denotes a point where the value at index $1 \leq j \leq \nu$ is $x_j$ if $j \in u$ and $z_j$ otherwise. The \emph{closed and total sensitivity indices} $\underline{s}_u = \underline{\tau}_u^2/\sigma^2$ and $\overline{s}_u = \overline{\tau}_u^2/\sigma^2$ respectively normalize the Sobol' indices to quantify the proportion of variance explained by a given subset of inputs. 

Suppose one is interested in computing the closed and total sensitivity indices of $\varphi$ at $u_1,\dots,u_c \subseteq 1:\nu$. Then we may choose the mean $\boldsymbol{\mu} \in \mathbb{R}^{(2, 3, c)}$ so that $\boldsymbol{\mu}_1,\boldsymbol{\mu}_2 \in \mathbb{R}^{(3,c)}$ contain values for the closed and total sensitivity indices respectively. Specifically,  $\boldsymbol{\mu}_{11},\boldsymbol{\mu}_{21} \in \mathbb{R}^c$ contain the closed and total Sobol' indices respectively while $\boldsymbol{\mu}_{i2}, \boldsymbol{\mu}_{i3} \in \mathbb{R}^c$ contain first and second moments respectively for any $i \in \{1,2\}$. For the QOI $\boldsymbol{s} \in \mathbb{R}^{(2, c)}$, we set $\boldsymbol{s}_1, \boldsymbol{s}_2 \in \mathbb{R}^c$ to contain the closed and total sensitivity indices respectively. 

Bounds may be propagated via $\boldsymbol{C}^-,\boldsymbol{C}^+:\mathbb{R}^{(2, 3, c)} \to \mathbb{R}^{(2, c)}$ defined for $i \in \{1,2\}$ and $j \in \{1,\dots,c\}$  by  
\begin{align*}
    C_{ij}^-(\boldsymbol{\mu}^-,\boldsymbol{\mu}^+) 
    = \begin{cases} 
        \text{clip}\left(\min\left(\frac{\mu_{i1j}^-}{\mu_{i3j}^+-\left(\mu_{i2j}^-\right)^2},\frac{\mu_{i1j}^-}{\mu_{i3j}^+-\left(\mu_{i2j}^+\right)^2}\right)\right), & \mu_{i3j}^- - \left(\mu_{i2j}^\pm\right)^2 >0 \\
        0, &\text{else}
     \end{cases} 
\end{align*}
with $C^+_{ij}(\boldsymbol{\mu}^-,\boldsymbol{\mu}^+)$ defined similarly and where $\text{clip}(\cdot) = \min(1,\max(0,\cdot))$ restricts values between 0 and 1. Above we have encoded the facts that sensitivity indices are between $0$ and $1$, the variance of $\varphi$ is non-negative, and Sobol' indices are non-negative. The dependency function $\boldsymbol{D}:\{\text{True},\text{False}\}^{(2, c)} \to \{\text{True},\text{False}\}^{(2, 3, c)}$ may be defined by broadcasting shapes so that for any $(1,1,1) \leq (i,j,k) \leq (2,3,c)$ we have $D_{ikj}(\boldsymbol{b}^{(\boldsymbol{s})}) = b_{ij}^{(\boldsymbol{s})}$. 

The QMCPy implementation further generalize to allow array objective functions $\boldsymbol{\varphi}: (0,1)^\nu \to \mathbb{R}^{\tilde{\boldsymbol{d}}_{\boldsymbol{\mu}}}$ so $\boldsymbol{d}_{\boldsymbol{\mu}} = (2,3,c,\tilde{\boldsymbol{d}}_{\boldsymbol{\mu}})$ and $\boldsymbol{d}_{\boldsymbol{s}} = (2,c,\tilde{\boldsymbol{d}}_{\boldsymbol{\mu}})$. Here the notation of nested vectors indicates that, for example, that  $(2,c,(5,6)) =(2,c,5,6)$. Also, notice that $d = 2 \nu$ in general. That is, the dimension of the node sequence is twice the size of the input dimension to $\varphi$.

Sensitivity indices present an illustrative case for computational complexity. Suppose the QMC algorithm takes $2^m$ total samples to accurately approximate all closed and total sensitivity indices for $u_1,\dots,u_c \subseteq 1:\nu$. Then the computational cost is $\$(\boldsymbol{\varphi})(2+c)2^m$ since every time our sensitivity index function is evaluated at $(\boldsymbol{x},\boldsymbol{z}) \in [0,1]^{2\nu}$ we must evaluate the users objective function at $\boldsymbol{x}$, $\boldsymbol{z}$, and $(\boldsymbol{x}_{u_j},\boldsymbol{z}_{-{u_j}})$ for $j=1,\dots,c$. If a user is only interested in approximating singleton sensitivity indices, $u_j = \{j\}$ for $j=1,\dots,\nu$, then it is possible to reduce the cost from $\$(\boldsymbol{\varphi})(2+\nu)2^m$ to $\$(\boldsymbol{\varphi})2^{m+1}$ using order $1$ replicated designs \cite{alex2008comparison,tissot2015randomized}. Such designs have been extended to  digital sequences in \cite{replicated_designs_sobol_seq} and utilized for sensitivity index approximation in \cite{reliable_sobol_indices_approx}.

A first example computes sensitivity indices of the Ishigami function \cite{ishigami1990importance} $g(\boldsymbol{T}) = (1+bT_3^4)\sin(T_1)+a\sin^2(T_2)$ where $\boldsymbol{T} \sim \mathcal{U}(-\pi,\pi)^3$ and $a=7$, $b=0.1$ as in \cite{crestaux2007polynomial,marrel2009calculations}. Figure \ref{SoRa_fig:ishigami} visualizes the resulting optimal approximations and QOI bounds which capture the exact sensitivity indices of the Ishigami function. 

\begin{figure}[t]
    \centering
    \includegraphics[width=.8\textwidth]{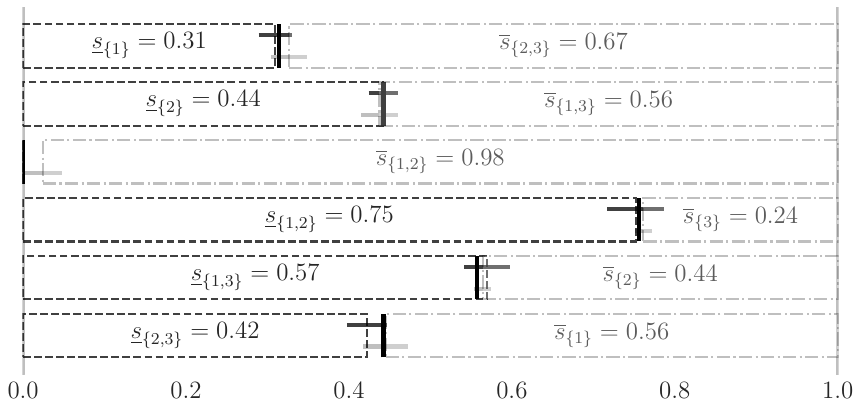}
    \caption{Approximate closed and total sensitivity indices for the Ishigami function illustrating the relationship $\underline{s}_u + \overline{s}_{u^c} = 1$ for all $u \subseteq 1:d$. In each row, the closed sensitivity index bar is  extended to the right from $0$ while the total sensitivity index bar is extended to the left from $1$. The bars should meet at the heavy vertical line for the analytic QOI $\underline{s}_u=1-\overline{s}_{u^c}$. The darker and lighter horizontal lines within each row depict the bounds for the closed and total sensitivity indices respectively. The heavy vertical line crossing both horizontal lines in each row indicates the true QOI is indeed captured in the bounds.}
    \label{SoRa_fig:ishigami}
\end{figure}

In another example, we compute sensitivity indices of a neural network classifier \cite{he2015delving} for the Iris dataset \cite{uci_ml_repo}. This example was inspired by a similar experiment in \cite{hoyt2021efficient}. The dataset consists of attributes sepal length (\textbf{SL}), sepal width (\textbf{SW}), petal length (\textbf{PL}), and petal width (\textbf{PW}), all in centimeters, from which an Iris is to be classified as either the \emph{setosa}, \emph{versicolor}, or \emph{virginica} species. We begin by fitting a neural network classifier that takes in input features and outputs a size $3$ vector of probabilities for each species summing to $1$. Taking the argument maximum among these three probabilities gives a species prediction. On a held out portion of the dataset, the neural network attains 98\% classification accuracy and may therefore be deemed a high quality surrogate for the true relation between input features and species classification. 

Our problem is to quantify, for each species, the variability in the classification probability attributed to a set of inputs. In other words, we would like to compute the sensitivity indices for each species probability. 
Here $\boldsymbol{d}_{\boldsymbol{\mu}} = (2,3,14,3)$ and $\boldsymbol{d}_{\boldsymbol{s}} = (2,14,3)$ since we have $3$ species classes, $14$ sensitivity indies of interest, and we are computing both the closed and total sensitivity indices. Figure \ref{SoRa_fig:nn_si} visualizes closed sensitivity index approximations. 

\begin{figure}[t]
    \centering
    \includegraphics[width=.8\textwidth]{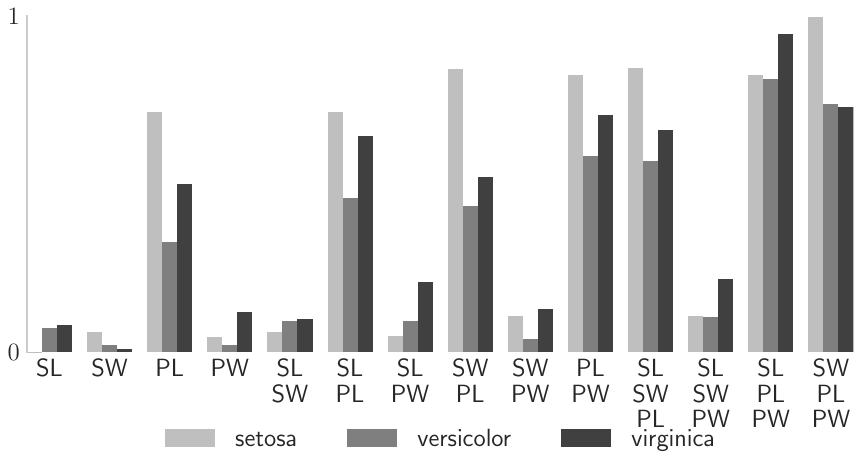}
    \caption{Closed sensitivity indices for neural network classification probability of each Iris species.}
    \label{SoRa_fig:nn_si}
\end{figure}

\section{Discussion and Further Work} \label{SoRa_sec:conclusions}

This article has utilized existing (Q)MC bounding techniques to approximate and bound array QOI formulated as a function of an array mean. The algorithm adaptively increases the sample size until a user specified stopping criterion on the QOI is met. The resulting bounds hold with uncertainty below a user specified threshold and the approximation is optimal with respect to the user specified error metric and error tolerance. Our work has been implemented into the open-source QMCPy package and exemplified on problems in machine learning and global sensitivity analysis. 

In the future, we hope to automatically analyze $\boldsymbol{C}$ to determine propagation functions $\boldsymbol{C}^-$ and $\boldsymbol{C}^+$ as well as the dependency structure encoded in $\boldsymbol{D}$. We also plan to allow dependency structures where QOI may depend on common individual solutions. Implementing order $1$ replicated designs will provide computational savings for special cases of sensitivity index computation as discussed in Section \ref{SoRa_sec:sensitivity_indices}. 

\section*{Acknowledgements}

The authors thank the referee for their valuable feedback. We would also like to thank Fred J. Hickernell for guidance and discussions which helped shape this work. 


\end{document}